\documentclass[12pt]{article}
\usepackage{subcaption}
\usepackage{color,soul}

\usepackage{amsmath}
\usepackage{amsthm}
\usepackage{amssymb}
\usepackage{cite}
\usepackage{url}
\usepackage[multiple]{footmisc}
\usepackage{graphicx}
\usepackage{epstopdf}
\usepackage{chngcntr}
\usepackage{enumitem}
\usepackage{hhline}
\usepackage{array}
\usepackage{hyperref}
\usepackage{color}

\hypersetup{colorlinks=false}

\if@twoside
\else 
\fi

\numberwithin{equation}{section}
\counterwithin{figure}{section}
\counterwithin{table}{section}

\theoremstyle{plain}
\newtheorem{theorem}{Theorem}[section]
\newtheorem{lemma}[theorem]{Lemma}
\newtheorem{proposition}[theorem]{Proposition}

\newtheorem{corollary}[theorem]{Corollary}

\theoremstyle{definition}
\newtheorem{definition}{Definition}[section]

\theoremstyle{remark}
\newtheorem*{remark}{Remark}

\newcommand{\norm}[1]{\left\|#1\right\|}
\newcommand{\abs}[1]{\left\vert#1\right\vert}
\newcommand{\spr}[1]{\left\langle\,#1\,\right\rangle}
\newcommand{\kl}[1]{\left(#1\right)}

\newcommand{\R}{\mathbb{R}} 

\newcommand{\N}{\mathbb{N}}

\newcommand{\mfunc}[1]{\text{#1} \,}

\newcommand{\sgn}{\mfunc{sgn}}

\newcommand{\xD}{x^\dagger}

\newcommand{\xkd}{x_k^\delta}

\newcommand{\yd}{y^{\delta}}

\newcommand{\eps}{\varepsilon}

\newcommand{\lt}{{\ell^2}}
\newcommand{\lo}{{\ell^1}}

\newcommand{\xade}{x_{\alpha,\eps}^\delta}
\newcommand{\Tad}{\mathcal{T}_{\alpha,\delta}}
\newcommand{\Jad}{\mathcal{J}_{\alpha,\delta}}
\newcommand{\Jade}{\mathcal{J}_{\alpha,\delta}^\eps}

\title{A note on the minimization of a Tikhonov functional with $\ell^1$-penalty}
\author{
Fabian Hinterer\footnote{Johannes Kepler University Linz, Institute of Industrial Mathematics, Altenbergerstra{\ss}e 69, A-4040 Linz, Austria (fabian.hinterer@indmath.uni-linz.ac.at), corresponding author.}, 
Simon Hubmer\footnote{Johann Radon Institute Linz, Altenbergerstra{\ss}e 69, A-4040 Linz, Austria, (simon.hubmer@ricam.oeaw.ac.at)}, 
Ronny Ramlau\footnote{Johannes Kepler University Linz, Institute of Industrial Mathematics, Altenbergerstra{\ss}e 69, A-4040 Linz, Austria, (ronny.ramlau@jku.at)} \footnote{Johann Radon Institute Linz, Altenbergerstra{\ss}e 69, A-4040 Linz, Austria, (ronny.ramlau@ricam.oeaw.ac.at)}
}

\begin{document}

\maketitle

\begin{abstract}
In this paper, we consider the minimization of a Tikhonov functional with an $\lo$ penalty for solving linear inverse problems with sparsity constraints. One of the many approaches used to solve this problem uses the Nemskii operator to transform the Tikhonov functional into one with an $\lt$ penalty term but a nonlinear operator. The transformed problem can then be analyzed and minimized using standard methods. However, by the nature of this transform, the resulting functional is only once continuously differentiable, which prohibits the use of second order methods. Hence, in this paper, we propose a different transformation, which leads to a twice differentiable functional that can now be minimized using efficient second order methods like Newton's method. We provide a convergence analysis of our proposed scheme, as well as a number of numerical results showing the usefulness of our proposed approach.

\end{abstract}


\section{Introduction}

In this paper, we consider linear operator equations of the form
	\begin{equation}\label{Ax=y}
		Ax = y \,,
	\end{equation}
where $A \, : \, \lt \to \lt$ is a bounded linear operator on the (infinite-dimensional) sequence space $\lt$. Note that by using a suitable basis or frame, operator equations between separable function spaces such as $L^p$, Sobolev, or Besov spaces can all be transformed into problems of the form \eqref{Ax=y}. We assume that only noisy data $\yd$ satisfying
	\begin{equation}\label{y-yd}
		\norm{y - \yd}_2 \leq \delta
	\end{equation}
are available, where $\norm{.}_2$ denotes the standard $\lt$-norm. Problems of the form \eqref{Ax=y} arise in many practical applications including, but not limited to, image processing (compression, denoising, enhancement, inpainting, etc.), image reconstruction, as well as medical and tomographic imaging. For example, in the case in tomography, where $A$ is the Radon transform and $x$ is the internal density to be reconstructed from sinogram data $\yd$, the solution $x$ can be expected to have a sparse representation in a given basis. Hence, we are particularly interested in sparse solutions of \eqref{Ax=y}, to which end we consider the minimization of the following Tikhonov functional
	\begin{equation}\label{Tikhonov} 
		\Tad(x) := \norm{Ax-\yd}_2^2 + \alpha \norm{x}_1 \,,
	\end{equation}
where $\norm{.}_1$ denotes the standard $\lo$-norm. This problem has already been thoroughly studied analytically (compare with Section~\ref{section_regularization}) as well as numerically (see Section~\ref{section_minimization} for an overview of previously proposed methods). However, the efficient minimization of the Tikhonov functional $\Tad$ still remains a field of active study, especially since the presence of the $\lo$-norm makes the functional non-differentiable at the origin. One approach to circumvent this issue was proposed in \cite{Ramlau_Teschke_2006}, where the authors considered a transformation of the Tikhonov functional into one which is once differentiable. In this paper, we extend their transformation idea by using an approximate transformation approach in order to end up with a functional that is also twice differentiable. This then allows the application of efficient second-order iterative methods for carrying out the minimization.

This paper is organized as follows: In Section~\ref{section_regularization}, we review known regularization results concerning sparsity regularization via the Tikhonov functional \eqref{Tikhonov} and in Section~\ref{section_minimization}, we discuss some of the existing methods for its minimization. In Section~\ref{transformation_approach}, we consider the transformation approach presented in \cite{Ramlau_Zarzer_2012} and its extension for obtaining twice differentiable functionals, for which we provide a convergence analysis. Furthermore, in Section~\ref{numerical_experiments}, we present numerical simulations based on a tomography problem to demonstrate the usefulness of our approach. Finally, a conclusion is given in Section~\ref{sect_conclusion}.

\section{Sparsity Regularization} \label{section_regularization}
In this section, we recall some basic results (adapted from \cite[Section~3.3]{Scherzer_Grasmair_Grossauer_Haltmeier_Lenzen_2008}) concerning the regularization properties of Tikhonov regularization with sparsity constraints. For a more extensive review on regularization theory for Tikhonov functionals with sparsity constraints the reader is referred to \cite{Resmerita_2005, Ramlau_Resmerita_2010, Jin_Maass_2012}, and more recently, \cite{ Jin_Maass_Scherzer_2017, Hohage_Sprung_Weidling_2020}.

First of all, concerning the well-definedness of minimizers of $\Tad$ and their stability with respect to the data $\yd$, we get the following result, which is an immediate consequence of \cite[Theorem~3.48]{Scherzer_Grasmair_Grossauer_Haltmeier_Lenzen_2008}: 
\begin{theorem}
Let $A \, : \, \lt \to \lt$ be weakly sequentially continuous, $\alpha > 0$ and $\yd \in \lt$. Then there exists a minimizer of the functional $\Tad$ defined in \eqref{Tikhonov}. Furthermore, the minimzation is weakly subsequentially stable with respect to the noisy data $\yd$.
\end{theorem}

Concerning the convergence of the minimizers of the Tikhonov functional, we get the following theorem, which follows directly from \cite[Theorem~3.49]{Scherzer_Grasmair_Grossauer_Haltmeier_Lenzen_2008}:
\begin{theorem}
Let $A \, : \, \lt \to \lt$ be weakly sequentially continuous, assume that the problem \eqref{Ax=y} has a solution in $\lo$, and let $\alpha(\delta) :  (0,\infty) \to (0, \infty) $ be chosen such that
	\begin{equation}\label{cond_alpha_delta}
		\alpha(\delta) \to 0 \,,
		\quad
		\text{and}
		\quad
		\frac{\delta^2}{\alpha(\delta)} \to 0 \,,
		\quad
		\text{as}
		\quad
		\delta \to 0 \,.
	\end{equation}
Moreover, assume that the sequence $\delta_k$ converges to $0$, that $y_k := y^{\delta_k}$ satisfies the estimate $\norm{y-y_k}_2\leq\delta_k$, and that $x_k$ is a sequence of elements minimizing $\mathcal{T}_{\alpha(\delta_k),y_k}$. Then there exists an $\lo$-minimum-norm solution $x^\dagger$ and a subsequence $x_{k_n}$ of $x_k$ such that $\norm{x_{k_n} - x^\dagger}_2 \to 0$ as $n \to \infty$. Furthermore, if the $\lo$-minimum-norm solution $x^\dagger$ is unique, then $\norm{x_k - x^\dagger}_2 \to 0$ as $k \to \infty$.
\end{theorem}

Note that typically, one only gets weak subsequential convergence of the minimizers of the Tikhonov functional to the minimum-norm solution. However, the above theorem shows that for sparsity regularization, one even gets strong subsequential convergence.

Furthermore, note that if $A$ is injective, the $\lo$-minimizing solution is sparse (i.e., only finitely many of its coefficients are non-zero) and satisfies a variational source condition, then it is possible to prove optimal convergence rates under the a-priori parameter choice $\alpha(\delta) \sim \delta$, both in Bregman distance and in norm \cite[Theorem~3.54]{Scherzer_Grasmair_Grossauer_Haltmeier_Lenzen_2008}.

\section{Minimization of the Tikhonov functional} \label{section_minimization}

In this section, we review some of the previously proposed methods for the minimization of \eqref{Tikhonov}. Due to the non-differentiability of the $\lo$-norm in zero, this minimization problem is a non-trivial task.  

Among the first and perhaps the most well-known method is the so-called \emph{Iterative Shrinkage Thresholding Algorithm (ISTA)}, proposed in \cite{Daubechies_Defrise_DeMol_2004}. Each iteration of this algorithm consists of a gradient-descent step applied to the residual functional, followed by a thresholding step, which leads to the iterative procedure
	\begin{equation}\label{ISTA}
		x_{k+1}^\delta  = S_{\alpha \omega} \kl{\xkd- \omega A^* \kl{A\xkd - y^\delta}} \, ,
	\end{equation}
where $S_{\alpha \omega}$ denotes the component-wise thresholding (shrinkage) operator
	\begin{equation*}
		\kl{S_{\alpha \omega}(x)}_k := \sgn(x_k) \max\{ \vert x_k \vert - \alpha \omega,0 \} \, . 
	\end{equation*}
It was shown that the iterates generated by ISTA converge to a minimizer of the Tikhonov functional \eqref{Tikhonov} under suitable assumptions  \cite{Daubechies_Defrise_DeMol_2004, Bredies_Lorenz_2008}. Unfortunately, this converge can be very slow, which motivated the introduction of \emph{Fast ISTA (FISTA)} in \cite{Beck_Teboulle_2009}. Based on Nesterov's acceleration scheme \cite{Nesterov_1983}, the iterates of FISTA are defined by
	\begin{equation}\label{FISTA}  
	\begin{split}
		x_{k}^\delta &=S_{\alpha \omega} \big( z_{k-1}^\delta- \omega A^* \big(Az_{k-1}^\delta-y^\delta \big) \big) \,,
		\qquad
		t_k = \tfrac{1+\sqrt{1+4 t_{k-1}^2}}{2} \,,
		\\
		z_k^\delta &= x_k^\delta+ \Big(\tfrac{t_{k -1}-1}{t_k} \Big) (x_k^\delta-x_{k-1}^\delta) \,,
		\qquad
		z_0^\delta =x_0 \,, \quad t_0 = 1 \,.
	\end{split}
	\end{equation} 
The convergence analysis presented in \cite{Beck_Teboulle_2009} as well as many numerical experiments show that the iterates of FISTA converge much faster than those of ISTA, the residual converging with a rate of $O(1/k^2)$ for FISTA compared to $O(1/k$) for ISTA, hence making it more practical. This speedup also holds for a generalized version of FISTA, which is applicable to composite (convex) minimization problems \cite{Attouch_Peypouquet_2016}.  Applied to problem \eqref{Tikhonov}, it has the same form as \eqref{FISTA}, but with the computation of $z_k^\delta$ replaced by
	\begin{equation*}
		z_{k}^\delta = x_k^\delta + \tfrac{k-1}{k + \beta -1} \kl{x_k^\delta - x_{k-1}^\delta} \,, 
	\end{equation*}
where the choice of $\beta = 3$ is common practice. The convergence of this method also for any other choice of $\beta > 3$ was established in \cite{Attouch_Peypouquet_2016}. 

In the context of compressed sensing, where one tries to recover signals from incomplete and inaccurate measurements in a stable way, minimization problems of the form \eqref{Tikhonov} have been analyzed and numerically treated in finite dimensions (see e.g.\ \cite{Candes_Romberg_Tao_2006, Donoho_Tanner_2005, Daubechies_DeVore_Fornasier_Gunturk_2010}). Also in finite dimensions, the minimization problem \eqref{Tikhonov} has been tackled sucessfully by using various Krylov-subspace techniques (see e.g.\ \cite{Buccini_Reichel_2019, Lanza_Morigi_Reichel_Sgallari_2015, Huang_Lanza_Morigi_Reichel_Sgallari_2017}).

In infinite dimensions, a number of different minimization algorithms for \eqref{Tikhonov} have been proposed. For example, the authors of \cite{Ramlau_Teschke_2006, Ramlau_Teschke_2005, Ramlau_Teschke_2010} have proposed a surrogate functional approach, while the authors of \cite{Bredies_Lorenz_Maass_2009, Bonesky_Bredies_Lorenz_Maass_2007} and \cite{Griesse_Lorenz_2008} have proposed conditional gradient and semi-smooth Newton methods, respectively. 

Of particular interest to us is the minimization approach presented in \cite{Ramlau_Zarzer_2012, Zarzer_2009}, which we discuss in detail in Section~\ref{transformation_approach} below. It is based on a nonlinear transformation utilizing a Nemskii operator, which turns the Tikhonov functional \eqref{Tikhonov} into one with a standard $\lt$-norm penalty, but with a nonlinear operator. Since the resulting transformed functional is continuously Fr\'echet differentiable, one can use standard first-order iterative methods for its minimization. Unfortunately, the functional is not twice differentiable, which prohibits the use of second-order methods, known for their efficiency. Circumventing this shortcoming is the motivation for the minimization approach based on an approximate transformation presented below.

\section{Transformation Approach} \label{transformation_approach} 
The concept of approximating a nonsmooth operator with a convergent sequence of smooth operators has been used before, e.g., in \cite{Acar_Vogel_1994} in the context of BV regularization. In the related setting where only an inexact forward operator is known, convergence of the resulting approximate solutions as the the uncertainty in the forward operator and the data decreases has been studied e.g., in \cite{Korolev_Lellmann_2018}.  
As described above, the authors of \cite{Ramlau_Zarzer_2012, Zarzer_2009} considered a transformation approach for minimizing the Tikhonov functional \eqref{Tikhonov}. This approach is based on a nonlinear transformation of the functional using the Nemskii operator
	\begin{equation}\label{def_N_p_q} 
	\begin{split}
		N_{p,q} \, : \,  (x_k)_{k \in \N}  \mapsto \kl{ \eta_{p,q}(x_k) }_{k \in \N} \,,
	\end{split}
	\end{equation}
where the function $\eta_{p,q}$ is defined by
	\begin{equation}\label{def_eta_p_q} 
		\eta_{p,q} \, : \, \R \to \R \,,
		\quad  		
		\tau \mapsto \sgn(\tau) \abs{\tau}^\frac{q}{p} \,.
	\end{equation} 
The operator $N_{p,q}$ has for example been used in the context of maximum entropy regularization \cite{Engl_Landl_1993}. Since here we need it only for the special case $p=1$ and $q=2$, we now define the operator
	\begin{equation}\label{def_N} 
	\begin{split}
		N \, : \, \lt \to \lo \,,
		\qquad
		x \mapsto N_{1,2}(x)  \,,
	\end{split}
	\end{equation}
and the function
	\begin{equation}\label{def_eta} 
		\eta \, : \, \R \to \R \,,
		\quad  		
		\tau \mapsto \eta_{1,2}(\tau) \,.
	\end{equation} 
The operator $N$ is continuous, bounded, bijective, and Fr{\'e}chet differentiable with
	\begin{equation} 
		N'(x)h = \kl{2 \abs{x_k}h_k }_{k \in \N} \,,
	\end{equation}
and is used to define the following nonlinear operator
	\begin{equation}\label{def_F}
		F \, : \, \lt \to \lt \,,
		\qquad
		x \mapsto (A \circ N)(x) \,.
	\end{equation}
This is then used to transform the problem of minimizing \eqref{Tikhonov} into a standard $\lt - \lt$ minimization problem, as shown by the following result from \cite{Ramlau_Zarzer_2012}:
\begin{proposition}
The following two problems are equivalent: 
\begin{enumerate}
\item Find $x^* \in \lo$, such that $x^*$ minimizes
	\begin{equation} \label{def_g}
		\Tad(x) = \norm{ Ax- \yd}_2^2 + \alpha \norm{x}_1 \,.
	\end{equation} 
\item Find $x^*= N(\tilde{x})$, such that $\tilde{x} \in \lt$ minimizes 
	\begin{equation} \label{def_Jad} 
		\Jad(x) :=  \norm{F(x) - \yd }_2^2 + \alpha \norm{x}_2^2 \,.
	\end{equation} 
\end{enumerate}
\end{proposition}

Due to the above proposition, both the original and the transformed problem recover the same solution, which thus have the same sparsity properties. Note that the operator $F$ is nonlinear even if $A$ is linear. However, using the transformed operator has the advantage that the resulting functional $\Jad$ is differentiable.

\begin{proposition}
The operator $F$ and the functional $\Jad$ defined in \eqref{def_F} and \eqref{def_Jad}, respectively, are continuously Fr{\'e}chet differentiable, with
	\begin{equation*} 
		F'(x)h = A N'(x)h \,,
		\qquad 
		\text{and}
		\qquad
		\Jad'(x) h= \spr{2F '(x)^*(F(x)-\yd) + 2\alpha x,h}  \,.
	\end{equation*}
\end{proposition}
\begin{proof}
This is an immediate consequence of the definition of $\Jad$ and the fact that $A$ is linear and $N$ is differentiable.
\end{proof}
Due to the above result, it is now possible to apply gradient based (iterative) methods for minimizing the transformed functional $\Jad$, and thus to compute a minimizer of the functional $\Tad$, which itself is not differentiable. 

Unfortunately, the transformed functional $\Jad$ is not twice differentiable, due to the fact that $N$ is not twice differentiable (at zero). This prohibits the use of second order methods like Newton's method, which are known to be very efficient in terms of iteration numbers. Hence, we propose to approximate $N$ by a sequence of operators $N_\eps$ which are twice continuously differentiable, and to minimize, instead of $\Jad$, the functional
	\begin{equation}\label{def_Jade}
		\Jade(x) := \norm{F_\eps(x) - \yd}_2^2 + \alpha \norm{x}_2^2 \,,
	\end{equation}
where we define the operator $F_\eps$ by
	\begin{equation}\label{def_F_eps}
		F_\eps \, : \, \lt \to \lt \,, 
		\quad 
		x \mapsto (A \circ N_\eps)(x) \,,
	\end{equation}
for a suitable approximation $N_\eps$ of the operator $N$. This approximation is based on suitable approximations $\eta_\eps$ of the functions $\eta$, which we introduce in the following
\begin{definition}
For $\eps > 0$ we define functions $\eta_\eps  : \R \to \R$ by
	\begin{equation}\label{def_eta_eps}
		\eta_\eps(\tau) := 
		\begin{cases} 
		-\tau^2 -\tfrac{1}{3} \eps^2 \,, &\tau \in  (-\infty, -\eps)  \,, \\
		\frac{1}{3\eps}\tau^3+\eps \tau \,,  &\tau \in [-\eps, \eps] \,, \\ 
		\tau^2+\tfrac{1}{3} \eps^2 \,,  &\tau \in  (\eps, \infty) \,.
		\end{cases}  
	\end{equation}
\end{definition}
\begin{figure} \centering \includegraphics[scale=0.6]{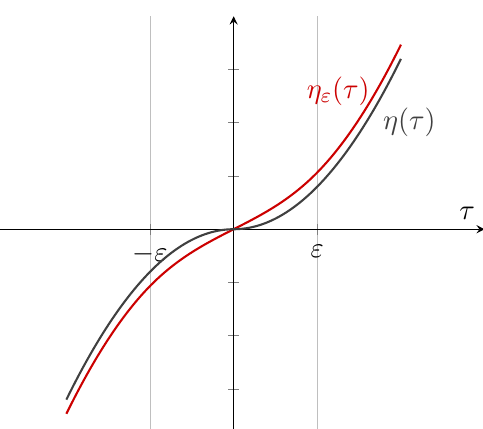}
\caption{Comparison of the transformation functions $\eta_\eps$ and $\eta$.}
\end{figure}
Obviously, $\eta_\eps \to \eta$ as $\eps \to 0$ and furthermore, we get the following
\begin{lemma}\label{lem_eta_eps_diff}
The functions $\eta_\eps$ defined by \eqref{def_eta_eps} are twice continuously differentiable. 
\end{lemma}
\begin{proof}
It follows from  its definition that $\eta_\eps$ is everywhere continuous and that
	\begin{equation*}
		\eta'_\eps(\tau) := 
		\begin{cases} 
		-2\tau \,, &\tau \in  (-\infty, -\eps)  \,, \\
		\frac{1}{\eps}\tau^2 + \eps \,,  &\tau \in [-\eps, \eps] \,, \\ 
		2 \tau\,,  &\tau \in  (\eps, \infty) \,.
		\end{cases}  
	\end{equation*} 
Again it follows that $\eta'_{\eps}$ is everywhere continuous and that
	\begin{equation*}
		\eta''_\eps(\tau) := 
		\begin{cases} 
		-2 \,, &\tau \in  (-\infty, -\eps)  \,, \\
		\frac{2}{\eps}\tau \,,  &\tau \in [-\eps, \eps] \,, \\ 
		2 \,,  &\tau \in  (\eps, \infty) \,,
		\end{cases}  
	\end{equation*} 	
which is again continuous everywhere, which concludes the proof.
\end{proof}

We now use the functions $\eta_\eps$ to build the operators $N_\eps$ via the following
\begin{definition}
For all $\eps > 0$ we define the operators
	\begin{equation}\label{def_N_eps}
		N_\eps  :  \lt \to \lt \,,
		\qquad 
		(x_k)_{k\in\N} \mapsto \kl{\eta_\eps(x_k)}_{k \in \N} \,.
	\end{equation}
\end{definition}

Concerning the well-defined and boundedness of $N_\eps$, we have the following
\begin{lemma}\label{lem_Neps_bounded}
The operators $N_\eps $ defined by \eqref{def_N_eps} satisfy
	\begin{equation}
		\norm{N_\eps(x)}_2 
		\leq
		\norm{x}_2 \sqrt{\tfrac{16}{9}\eps^2+ 2 \norm{x}_2^2}  \,,
	\end{equation}
and are therefore well-defined as operators from $\lt \to \lt$.
\end{lemma}
\begin{proof}
Let $\eps > 0$ be arbitrary but fixed and take $x = (x_k)_{k \in \N} \in \lt$. We have that
	\begin{equation*}
	\begin{split}
		\abs{\eta_\eps(x_k)} &= 
		\begin{cases} 
		\abs{x_k}^2 + \tfrac{1}{3} \eps^2 \,, & \abs{x_k} > \eps  \,, \\
		\frac{1}{3 \eps}\abs{x_k}^3 +\eps\abs{x_k}  \,,  &\abs{x_k} \leq \eps \,, \\ 
		\end{cases}  
		\\ \vspace{2pt} \\
		& \leq
		\begin{cases} 
		\abs{x_k}^2 + \tfrac{1}{3} \eps \abs{x_k} \,, & \abs{x_k} > \eps  \,, \\
		\frac{4}{3}\eps\abs{x_k} \,,  &\abs{x_k} \leq \eps \,. \\ 
		\end{cases}  
	\end{split}	
	\end{equation*}
Therefore, we get that
	\begin{equation*}
	\begin{split}
		\norm{N_\eps(x)}_2 ^2
		& =
		\sum\limits_{k \in \N} \abs{\eta_\eps(x_k)}^2
		= 
		\sum\limits_{\abs{x_k} \leq \eps} \abs{\eta_\eps(x_k)}^2
		+
		\sum\limits_{\abs{x_k} > \eps} \abs{\eta_\eps(x_k)}^2
		\\
		& \leq 
		\kl{\tfrac{4}{3} \eps}^2 \sum\limits_{\abs{x_k} \leq \eps} \abs{x_k}^2
		+
		\sum\limits_{\abs{x_k} > \eps} \kl{\abs{x_k}^2 + \tfrac{1}{3} \eps \abs{x_k}}^2
		\\
		& \leq 
		\kl{\tfrac{4}{3} \eps}^2 \sum\limits_{\abs{x_k} \leq \eps} \abs{x_k}^2
		+
		2 \sum\limits_{\abs{x_k} > \eps} \abs{x_k}^4 + \tfrac{2}{9} \eps^2 \sum\limits_{\abs{x_k} > \eps} \abs{x_k}^2 \,,
	\end{split}
	\end{equation*}
from which we derive that
	\begin{equation*}
	\begin{split}
		\norm{N_\eps(x)}_2 ^2
		& 
		\leq 
		\tfrac{16}{9} \eps^2  \sum\limits_{k = 1}^\infty \abs{x_k}^2
		+
		2 \sum\limits_{k=1}	^\infty \abs{x_k}^4\
		\\
		&=
		\tfrac{16}{9} \eps^2  \norm{x}_2^2 
		+
		2 \norm{x}_4^4
		\leq
		\kl{\tfrac{16}{9}\eps^2+ 2 \norm{x}_2^2} \norm{x}_2^2 \,,
	\end{split}
	\end{equation*}
which immediately yields the assertion.
\end{proof}

The operators $N_\eps$ are also continuous, as we see in the following 
\begin{proposition}\label{prop_N_eps_cont}
The operators $N_\eps \, : \, \lt \to \lt$ defined by \eqref{def_N_eps} are continuous.
\end{proposition}
\begin{proof}
Let $\eps > 0$ and $x = (x_k)_{k\in\N} \in \lt$ be arbitrary but fixed, and consider a sequence $x^n = (x^n_k)_{k \in \N} \in \lt$ converging to $x$. It follows that the norm of $x^n$ is uniformly bounded, i.e., there exists a constant $c > 0$ such that $\norm{x^n} \leq c$ for all $n$, from which it also follows that $\abs{x^n_k} \leq c$ for all $k$ and $n$. Furthermore, since the function $\eta_\eps$ is continuously differentiable, it follows that it is Lipschitz continuous on bounded sets. This implies that there exists a Lipschitz constant $L> 0$ such that
	\begin{equation}
		\abs{\eta_\eps(x^n_k) - \eta_\eps(x_k)} \leq L \abs{x^n_k - x_k } \,. 	
	\end{equation}	 
Hence, we get that
	\begin{equation}
		\norm{ N_\eps(x^n) - N_\eps(x) }_2^2
		=
		\sum\limits_{k=1}^\infty \abs{\eta_\eps(x^n_k) - \eta_\eps(x_k)}^2 
		 \leq 
		L^2 \sum\limits_{k=1}^\infty \abs{x^n_k - x_k }^2
		= 
		L^2 \norm{ x^n - x}_2^2   \,,
	\end{equation}
and therefore,
	\begin{equation}
		\norm{ N_\eps(x^n) - N_\eps(x) }_2
		\leq 
		L \norm{ x^n - x}_2   
		\quad 
		\to 0
		\qquad 
		\text{as} \quad n \to \infty \,,
	\end{equation}
which shows the continuity of $N_\eps$ and concludes the proof.
\end{proof}

By their construction, the operators $N_\eps$ are also twice differentiable, as we see in
\begin{proposition}\label{prop_N_eps_diff}
The operators $N_\eps \, : \, \lt \to \lt $ defined by \eqref{def_N_eps} are twice continuously Fr\'echet differentiable, with 
	\begin{equation}
		N'_\eps (x)h = \kl{ \eta_\eps ' (x_k ) h_k  }_{k \in \N} \,,
		\qquad 
		\text{and}
		\qquad
		N''_\eps (x)(h,w) =\kl{ \eta_\eps '' (x_k ) h_k w_k }_{k \in \N}  \,. 
	\end{equation}
\end{proposition}
\begin{proof}
This follows from the definition of $N_\eps$ together with Lemma~\ref{lem_eta_eps_diff}. 
\end{proof}

The approximation properties of the operators $N_\eps$ are studied in the following 
\begin{proposition}\label{prop_N_approx}
For $N$ and $N_\eps$ be defined by \eqref{def_N} and \eqref{def_N_eps}, respectively, it holds that
	\begin{equation}
		\norm{N(x) - N_\eps(x)}_2 \leq \tfrac{7}{3} \eps \norm{x}_2 \,.
	\end{equation}
\end{proposition}
\begin{proof}
Let $\eps > 0$ and  $x \in \lt$ be arbitrary but fixed. Then it holds that
	\begin{equation*}
		\eta_\eps(x_k) - \eta(x_k) 
		=
		\begin{cases}
		- \tfrac{1}{3} \eps^2	 \,, & x_k \in (-\infty,-\eps) \,,
		\\
		\tfrac{1}{3\eps}x_k^3 + \eps x_k  + x_k^2 \,, & x_k \in [-\eps,0] \,,
		\\
		\tfrac{1}{3\eps}x_k^3 + \eps x_k - x_k^2 \,, & x_k \in [0,\eps] \,,
		\\
		\tfrac{1}{3}\eps^2 \,, & x_k \in (\eps,\infty) \,,
		\end{cases}
	\end{equation*}
from which it follows that
	\begin{equation*}
	\begin{split}
		\abs{\eta_\eps(x_k) - \eta(x_k) }
		 &=
		\begin{cases}
		\tfrac{1}{3} \eps^2	 \,, & \abs{x_k} > \eps \,,
		\\
		\abs{  \tfrac{1}{3\eps} \abs{x_k}^3 + \eps \abs{x_k} - \abs{x_k}^2    } 		
		\,, &  \abs{x_k} \leq \eps\,.
		\end{cases}
		\\ \vspace{2pt} \\
		&\leq
		\begin{cases}
		\tfrac{1}{3} \eps \abs{x_k}	 \,, & \abs{x_k} > \eps \,,
		\\
		 \tfrac{1}{3} \eps \abs{x_k} + \eps \abs{x_k} + \eps \abs{x_k}   		
		\,, &  \abs{x_k} \leq \eps\,,
		\end{cases}.
	\end{split}
	\end{equation*}
and therefore
	\begin{equation*}
		\abs{\eta_\eps(x_k) - \eta(x_k) }
		\leq 
		\tfrac{7}{3} \eps \abs{x_k} \,.
	\end{equation*}
This now implies that
	\begin{equation*}
	\begin{split}
		\norm{N_\eps(x) - N(x) }_2^2
		= 
		\sum\limits_{k=1}^\infty \abs{\eta_\eps(x_k) - \eta(x_k)}^2
		= \kl{\tfrac{7}{3}\eps}^2
		\sum\limits_{k=1}^\infty \abs{x_k}^2 
		= \kl{\tfrac{7}{3}\eps}^2 \norm{x}_2^2 \,,
	\end{split}
	\end{equation*}
from which the statement immediately follows.
\end{proof}

The above result immediately implies an approximation result for the operators $F_\eps$. 
\begin{corollary}\label{cor_F_approx}
Let $A \, : \, \lt \to \lt$ be a bounded and linear operator and let $F$ and $F_\eps$ be defined by \eqref{def_F} and \eqref{def_F_eps}, respectively. Then it holds that
	\begin{equation}
		\norm{F(x) - F_\eps(x)}_2 \leq \tfrac{7}{3} \eps \norm{A} \norm{x}_2 \,.
	\end{equation}
\end{corollary}
\begin{proof}
By the definition of $F$ and $F_\eps$, we have that
	\begin{equation*}
	\begin{split}
		\norm{F(x) - F_\eps(x)}_2 
		=
		\norm{(A \circ N)(x) - (A \circ N_\eps)(x) }_2
		\leq 
		\norm{A} \norm{N(x) - N_\eps(x)}_2 \,,
 	\end{split}
	\end{equation*}
which, together with Proposition~\ref{prop_N_approx} now yields the assertion.
\end{proof}

Other important properties of the operators $F$ and $F_\eps$ are collected in the following 
\begin{proposition}\label{prop_Fe_comp_closed}
Let $A : \lt \to \lt $ be a bounded linear operator. Then the operators $F$ and $F_\eps$ defined by \eqref{def_F} and \eqref{def_F_eps}, respectively, are continuous and weakly sequentially closed.
\end{proposition}
\begin{proof}
Since $A$ and, due to Proposition~\ref{prop_N_eps_cont}, $N_\eps$ are continuous, by its definition also $F_\eps$ is continuous. In order to show the weak sequential closedness of $F_\eps$, note that since its definition space is the whole of $\lt$, it suffices to show that $F_\eps$ is weakly continuous. For this, take an arbitrary sequence $x^n \in \lt$ converging weakly to some element $x \in \lt$.  Since in $\lt$ a sequence converges weakly if and only if it converges componentwise and its norm is bounded \cite{Conway_1994}, it follows from the continuity and boundedness of $N_\eps$ (Lemma~\ref{lem_Neps_bounded}) and Proposition~\ref{prop_N_eps_cont}) that $N_\eps(x^n)$ converges weakly to $N_\eps(x)$. Now, as a bounded linear operator, $A$ is also weakly sequentially continuous. Hence, since $F_\eps = A \circ N_\eps$, it follows that $F_\eps(x^n)$ converges weakly to $F_\eps(x)$, which establishes its weak sequential continuity and consequentially also its weak sequential closedness.  For the operator $F$, these result have already been shown in \cite{Ramlau_Zarzer_2012}. However, noting that Lemma~\ref{lem_Neps_bounded} and Proposition~\ref{prop_N_eps_cont} also hold for the limit case $\eps = 0$, they also follow the same way as above.
\end{proof}

Furthermore, the differentiability of $N_\eps$ immediately translates into the following
\begin{proposition} \label{derivatives}
The operators $F_\eps$ and thus the functionals $\Jade$ defined in \eqref{def_F_eps} and \eqref{def_Jade}, respectively, are twice continuously Fr{\'e}chet differentiable, where
	\begin{equation*} 
	\begin{split}
		&F_\eps'(x)h = A N_\eps'(x)h \,,
		\qquad 
		F_\eps''(x)(h,w) = A N_\eps''(x)(h,w) \,,
		\\
		&\Jade\, '(x)h = 2\spr{F_\eps '(x)^*(F_\eps(x)-\yd) + \alpha x ,h} \,,
		\\
		&\Jade\,''(x)(h,w) =  2 \spr{F_\eps(x) - \yd, F_\eps''(x)(h,w) } +  2\spr{F_\eps '(x)^*F_\eps'(x)w   + \alpha w ,h}   \,.
	\end{split}	
	\end{equation*}
\end{proposition}
\begin{proof}
This follows from the definition of $F_\eps$ and $\Jade$ together with Proposition~\ref{prop_N_eps_diff}.
\end{proof}
We now consider the problem of minimizing the Tikhonov functional $\Jade$, whose minimizers we denote by $\xade$. Due to the above results, the classical analysis of Tikhonov regularization for nonlinear operators is applicable (see for example \cite{Engl_Hanke_Neubauer_1996, Engl_Ramlau_2015}), and we immediately get the following
\begin{theorem}
Let $A\,:\, \lt \to \lt$ be a bounded, linear operator and let $F_\eps$ be defined by \eqref{def_F_eps}.Then for each $\alpha > 0$, a minimizer $\xade$ of the functional $\Jade$ defined in \eqref{def_Jade} exists. Furthermore, the minimization of $\Jade$ is stable under perturbations of $\yd$. 
\end{theorem}
\begin{proof}
Since by Proposition~\ref{prop_Fe_comp_closed}, the operator $F_\eps$ is continuous and weakly sequentially closed, this follows immediately from \cite[Theorem~10.2]{Engl_Hanke_Neubauer_1996}.
\end{proof}

Next, we are interested in the behaviour of the minimizers $\xade$ as $\eps \to 0$. Given a suitable coupling of the noise level $\delta$ and the parameter $\eps$, we get the following
\begin{theorem} \label{conv}
Assume that $F(x) = y$ has a solution and let $\alpha(\delta)$ and $\eps(\delta)$ satisfy
	\begin{equation}\label{cond_alpha_delta_eta}
		\alpha(\delta) \to 0 \,,
		\quad
		\eps(\delta) \to 0 \,,
		\quad
		\frac{\delta^2}{\alpha(\delta)} \to 0 \,,
		\quad
		\frac{\eps^2}{\alpha(\delta)} \to 0 \,,
		\quad
		\text{as}
		\quad
		\delta \to 0 \,.
	\end{equation}
Then $x_{\alpha(\delta),\eps(\delta)}^\delta$ has a convergent subsequence. Moreover, the limit of every convergent subsequence is a minimum-norm solution of $F(x) = y$. Furthermore, if the minimum-norm solution $\xD$ is unique, then
	\begin{equation}
		\lim\limits_{\delta \to 0} x_{\alpha(\delta),\eps(\delta)}^\delta \, = \, x^\dagger \,.
	\end{equation}
\end{theorem}
\begin{proof}
The proof of this theorem follows the same lines as the classical proof of convergence of Tikhonov regularization \cite{Engl_Hanke_Neubauer_1996} and the proof for the case that the operator is approximated by a series of finite dimensional operators \cite{Neubauer_1989, Poeschl_Resmerita_Scherzer_2010} (in which case a slightly stronger condition than what we can derive from Proposition~\ref{prop_N_approx} was used). Hence, we here only indicate the main differences in the proof.

Note first that due to Proposition~\ref{prop_N_approx}, it follows that
	\begin{equation}
	\begin{split}
		\norm{F_\eps(x) - F(x)  }_2 
		\leq
		\norm{A}\norm{N_\eps(x) - N(x)  }_2   
		\leq
		\tfrac{7}{3} \eps \norm{A}\norm{x}_2  \,.
	\end{split}
	\end{equation}
This, together with $\xade$ being a minimizer of $\Jade$ implies that
	\begin{equation}\label{eq_helper_3}
	\begin{split}
		\norm{F_\eps(\xade) - \yd }_2^2 + \alpha \norm{\xade}_2^2 
		&\leq
		\norm{F_\eps(\xD) - \yd }_2^2 + \alpha \norm{\xD}_2^2 
		\\
		&\leq
		\kl{\tfrac{7}{3} \norm{A}\norm{\xD}_2  \eps + \delta}^2 + \alpha \norm{\xD}_2^2 \,.
	\end{split}
	\end{equation}
Together with \eqref{cond_alpha_delta_eta}, this implies the boundedness of $\xade$ and 
	\begin{equation*}
		\lim\limits_{\delta \to 0} \norm{F_\eps(\xade) - \yd }_2 = 0 \,.
	\end{equation*}
Hence, since then there holds
		\begin{equation*}
		\begin{split}
		\norm{F(\xade) - y }_2 
		&\leq
		\norm{F_\eps(\xade) - \yd }_2 
		+
		\norm{F_\eps(\xade) - F(\xade) }_2 
		+
		\norm{y - \yd}_2
		\\
		&\leq
		\norm{F_\eps(\xade) - \yd }_2 
		+
		\delta
		+
		\tfrac{7}{3}\norm{A}\norm{\xade}_2	\eps 	
		 \quad
		\underset{\delta \to 0} {\longrightarrow}
		\quad 
		0 \,,
	\end{split}
	\end{equation*}
the weak sequential closedness of $F$ implies the convergence of a subsequence of $\xade$ to a solution of $F(x) = y$. The remainder of the proof then follows analogously to the one of \cite[Theorem~10.3]{Engl_Hanke_Neubauer_1996} and is therefore omitted here.
\end{proof}

The above result shows that minimizing $\Jade$ instead of $\Jad$ to approximate the solution of $F(x) = y$ makes sense if $\eps$ and the noise level $\delta$ are suitably coupled, for example via $\eps \sim \delta$. Furthermore, the assumption that $F(x) = y$ solvable, is for example satisfied if $Ax = y$ has a solution belonging not only to $\lt$ but also to $\lo$, i.e., is sparse.

\begin{remark}
Following the line of the proofs of classical Tikhonov regularization results, it is also possible to derive convergence rate results under standard assumptions. Furthermore, the above analysis also holds for nonlinear operators $A$ which are Lipschitz continuous, since then Corollary~\ref{cor_F_approx} also holds. 
\end{remark}

\section{Minimization methods for the Tikhonov functional} \label{numerical_experiments} 

In the previous section, we established existence, stability, and convergence of the minimizers of $\Jad$ and $\Jade$ under standard assumptions. However, there still remains the question of how to actually compute those minimizers in an efficient way.

One way to do this is to interpret the minimization of $\Jad$ and $\Jade$ as Tikhonov regularization for the nonlinear operator equations $F(x) = y$ and $F_\eps(x) = y$, respectively, and to use iterative regularization methods for their solution. Since both the operator $F$ and $F_\eps$ are continuously Fr\'echet differentiable, iterative regularization methods like Landweber iteration \cite{Kaltenbacher_Neubauer_Scherzer_2008}, TIGRA \cite{Ramlau_2003}, the Levenberg-Marquardt method \cite{Hanke_1997,Jin_2010} or iteratively regularized Gauss-Newton  \cite{Blaschke_Neubauer_Scherzer_1997,Jin_Tautenhahn_2009} are applicable. Of course, as all of those methods only require a once differentiable operator, it makes sense in terms of accuracy to apply them for the operator $F$ and not for the approximated operator $F_\eps$.

Another way is to use standard iterative optimization methods for the (well-posed) problem of minimizing $\Jad$ or $\Jade$. In particular, since we have derived in the previous section that $\Jade$ is twice continuously Fr\'echet differentiable, efficient second order methods like Newton's method are applicable for its minimization.

In this section, we introduce and discuss some details of the minimization methods used to obtain the numerical results presented in Section~\ref{sect_numerics} below.

\subsection{Gradient descent, ISTA and FISTA}

We have seen that the Tikhonov functional $\Jad$ defined in \eqref{def_Jad} is continuously Fr\'echet differentiable. Hence, it is possible to apply gradient descent for its minimization. 

For this, note first that since $N'(x)h$ is a linear operator, it can be written as
	\begin{equation}\label{eq_N_G}
		N'(x)h = G(x)h \,,
	\end{equation}
where $G(x)$ is the infinite dimensional `matrix' representation of $N'(x)$ given by
	\begin{equation*}
		G(x) := \text{diag}(2\abs{x_k})_{k \in \N} \,,
	\end{equation*}
which is called the $\emph{gradient}$ of $N$. Similarly, there is an (infinite-dimensional) matrix representation of $\Jad'(x)$, i.e., the gradient $\nabla \Jad(x)$ of $\Jad(x)$, which is given by
	\begin{equation*}
		\nabla \Jad (x) := 2 G(x) A^T \kl{AN(x)-y^\delta} + 2 \alpha x \,,
	\end{equation*}
where, with a small abuse of notation, $A$ denotes the (infinite-dimensional) matrix representation of the linear operator $A$, and $A^T$ denotes its transpose.

Using the above representations, we can now write the gradient descent algorithm for minimizing $\Jad$ in the well-known form
	\begin{equation}\label{gradient_descent} 
		x_{n+1}^\delta  = x_n^\delta  - \omega_n \nabla \Jad (x_n^\delta ) \, , 
	\end{equation} 
where $\omega_n$ is a sequence of stepsizes. If the stepsizes are chosen in a suitable way, for example via the Armijo rule \cite{Hinze_Ulbrich_Ulbrich_2009}, the iterates converge to a stationary point of $\Jad$ (see e.g.\ \cite[Theorem 2.2]{Hinze_Ulbrich_Ulbrich_2009}). In order to stop the iteration, we employ the well-known \emph{discrepancy principle}, i.e., the iteration is terminated with index $n_* = n_*(\delta,\yd)$, when for the first time 
	\begin{equation}\label{discrepancy_nonlinear}
		\norm{F(x_{n_*}^\delta )-y^\delta}_2 \leq \tau \delta  \,,
	\end{equation}
where $\tau>1$ is fixed. Note that since the Tikhonov functional may have several (local and global) minima, convergence to a global minimum is only guaranteed if a sufficiently good initial guess is chosen.

The (infinite-dimensional) matrix representations introduced above can also be used to rewrite ISTA \eqref{ISTA} into the following form 
	\begin{equation*}
		x_{n+1}^\delta  = S_{\alpha \omega} \kl{x_n^\delta - \omega \, 2 \, G(x_n^\delta ) A^T \kl{AN(x_n^\delta )-y^\delta}} \, ,
	\end{equation*}
which immediately also translates to a similar rewriting of FISTA defined in \eqref{FISTA}.

\subsection{The Levenberg-Marquardt method}  

It is well-known that gradient based methods like gradient descent or ISTA are quite slow with respect to convergence speed. Although it is possible to speed them up by using suitable stepsizes (see for example \cite{Saxenhuber_2016,Neubauer_2017_2}) or acceleration schemes like FISTA, it is often advantageous to use second-order methods instead. One such method is the Levenberg-Marquardt method \cite{Hanke_1997,Jin_2010}, which is given by 
	\begin{equation}\label{Levenberg_Marquardt} 
		x^\delta_{n+1}=x^\delta_n + \kl{ F'(x^\delta_n)^* F'(x^\delta_n)+\alpha_n  I }^{-1}F'(x^\delta_n)^*\kl{ y^\delta - F(x^\delta_n) }.
	\end{equation}
Although this is a second-order method, it only requires the operator $F$ to be once continuously Fr\'echet differentiable. Using again the (infinite-dimensional) matrix representation of $N'(x)h$ from \eqref{eq_N_G}, the method can be rewritten into the following form
	\begin{equation*}
		x_{n+1}^\delta = x_n^\delta  + \kl{ G(x_n^\delta) A^T A G(x_n^\delta) + \alpha_n I }^{-1} G(x_n^\delta) A^T(y^\delta - F(x_n^\delta))  \, .
	\end{equation*}
In order to obtain convergence of this method, one needs, among other things, a suitably chosen sequence $\alpha_n$ converging to $0$ as well as a sufficiently good initial guess \cite{Hanke_1997}). As a stopping rule, one usually also employs the discrepancy principle \eqref{discrepancy_nonlinear}.

The Levenberg-Marquardt method typically requires only very few iterations to satisfy the discrepancy principle. However, in each iteration step the linear operator $\kl{ F'(x^\delta_n)^* F'(x^\delta_n)+\alpha_n  I }$ has to be inverted, which might be costly for some applications. This can be circumvented, though, via approximating the result of this inversion by the application of number of iterations of the conjugate gradient method.

It is possible to add an additional regularization term to the Levenberg-Marquardt method, thereby ending up with the so-called \emph{iteratively-regularized Gauss-Newton method} \cite{Blaschke_Neubauer_Scherzer_1997,Jin_Tautenhahn_2009}. Typically behaving very similar in practice, this method can be proven to converge under slightly weaker assumptions than the Levenberg-Marquardt method.

\subsection{Newton's method}\label{sect_Newton}

In contrast to $\Jad$, the functional $\Jade$ is twice continuously Fr\'echet differentiable. The information contained in this second derivative can be used to design efficient methods for its minimization. One such method, based on Newton's method, is considered here.

Note that the first-order optimality condition for minimizing $\Jade$ is given by
	\begin{equation}\label{optimality_condition} 
		 \Jade \,'(x)h= 0  \qquad \forall\,  h \in \lt \,.
	\end{equation} 
Using Taylor approximation in the above equation yields
	\begin{equation*} 
		\Jade\,' (x + \tau)(h) = \Jade\,' (x)h + \Jade\,'' (x)(\tau,h) \qquad \forall \, h \in \lt \,,
	\end{equation*} 
which, for the special choice of $x = x_n$ and $\tau = (x_{n+1}-x_n)$, becomes
	\begin{equation}\label{newton}
		 \Jade\,' (x_n)(h)  + \Jade\,''(x_n)(x_{n+1}-x_n,h) = 0 \, \qquad \forall \, h \in \ell^2  \,.
	\end{equation}  
This implicitly defines an iterative procedure,
which is nothing else than Newton's method applied to the optimality condition \eqref{optimality_condition}. Since $\Jade\,''(\cdot,h)$ is continuously invertible around the global minimizer, this method is (locally) well-defined and q-superlinearly convergent (see for example \cite[Corollary~2.1]{Hinze_Ulbrich_Ulbrich_2009}). 

We can again use an (infinite-dimensional) matrix representation to rewrite this iterative procedure into a more familiar form. For this, we first define the `matrices'
	\begin{equation}
		G_\eps(x) := \text{diag}(\eta_\eps' (x_k) )_{k \in \N} \,, 
		\qquad 
		H_\eps(x,w) := \text{diag}(\eta_\eps'' (x_k) w_k )_{k \in \N} \,,
	\end{equation}
which correspond to the gradient and the Hesse matrix of $N_\eps(x)$, and use this to write
	\begin{equation}
		N_\eps'(x) h = G_\eps(x) h  \,,
		\qquad 	
		N_\eps''(x) (w,h) = H_\eps(x,w) h \,.
	\end{equation}
This allows the following matrix representation of the functionals $\Jade \,'(x)$ and $\Jade\,''(x)$
	\begin{equation*}
		\nabla \Jade (x) := 2 G_\eps (x) A^T \kl{AN_\eps (x) - y^\delta }  + 2 \alpha x \,,
	\end{equation*}
	\begin{equation*}
		\nabla^2 \Jade (x) := 2 H_\eps \kl{x,A^T \kl{AN_\eps (x) - y^\delta } } + 2 G_\eps(x) A^T A G_\eps(x) + 2 \alpha I  \,, 
	\end{equation*}
where $I$ denotes the identity matrix, and $\nabla \Jade (x)$ and $\nabla^2 \Jade(x)$ can be seen as the gradient and the Hessian matrix of the functional $\Jade$, respectively. Using these representations, the iterative procedure \eqref{newton} can be rewritten into the more familiar form
	\begin{equation*}
		\nabla \Jade (x_n) + \nabla^2 \Jade (x_n) (x_{n+1}-x_n) = 0 \, .
	\end{equation*}
which is an infinite-dimensional matrix-vector system for the update $(x_{n+1} - x_n)$.

\section{Numerical Examples}\label{sect_numerics}

In this section, we demonstrate the usefulness of our proposed approximation approach on a numerical example problem based on \emph{Computerized Tomography (CT)}. In particular, we focus on how the Newton approach for the minimization of $\Jade$ introduced in Section~\ref{sect_Newton} above performs in comparison to the other methods presented in Section~\ref{section_minimization}.

In the medical imaging problem of CT, one aims to reconstruct the density function $f$ inside an object from measurements of the intensity loss of an X-ray beam sent through it. In the 2D case, for example if one scans a cross-section of the human body, the relationship between the intensity $I_0$ of the beam at the emitter position and the intensity $I_L$ at the detector position is given by \cite{Natterer_2001}
	\begin{equation}\label{tomography_equation} 
		\log I_L(s,w) - \log I_0(s,w) = - 		\int_\R f(sw+tw^\perp ) \, dt \,.
	\end{equation} 
Thus, if one defines the well-known \emph{Radon transform} operator
	\begin{equation*} 
		Rf(s,w) := \int_\R f(sw+tw^\perp ) \, dt \,, 
	\end{equation*} 
the reconstruction problem \eqref{tomography_equation} can be written in the standard form 
	\begin{equation*}
		R f = g \,.
	\end{equation*}
Expressing $f$ in terms of some basis or frame, and noting that typically one considers objects whose density is equal to $0$ on large subparts, the above problem precisely fits into the framework of $\lo$ sparsity regularization considered in this paper.

\subsection{Discretization and Implementation}

In order to obtain a discretized version of problem \eqref{tomography_equation}, we make use of the toolbox AIR TOOLS II by Hansen and Jorgensen \cite{Hansen_Jorgensen_2017}. Therein, the density function $f$ is considered as a piecewise constant function on an $m\times m$ pixel grid (see Figure~\ref{fig_allmethods} for examples). With this, equation \eqref{tomography_equation} can be written in the discretized form
 	\begin{equation}\label{eq_discr}
 		y_i := -\kl{\log I_L^{(i)} - \log I_0^{(i)} } = \sum_{j=1}^{m^2} a_{ij}x_j  
 	\end{equation} 
where the $x_j$ denote the value of $f$ at the $j$-th pixel, $ I_0^{(i)}$ and $I_L^{(i)}$ denote the emitted and detected intensity of the $i$-th ray, respectively, and $a_{ij}$ denotes the length of the path which it travels through within the $j$-th pixel cell. Note that since any given ray only travels through relatively few cells, most of the coefficients $a_{ij}$ are equal to $0$ and thus the matrix $A$ is sparse. Collecting the coefficients $a_{ij}$ into a matrix $A$, equation \eqref{eq_discr} can be written as a matrix-vector equation of the form
	 \begin{equation*} 
	 	Ax = y \, .
	 \end{equation*}

Specifying all required parameters as well as the exact solution which one wants to reconstruct, the toolbox provides both the matrix $A$ and the right-hand side vector $y$. For our purposes, we used the toolbox function \texttt{paralleltomo}, creating a parallel beam tomography problem with (the suggested default values of) $180$ angles and $70$ parallel beams for each of them. For the number of pixels we used $m^2 = 50^2$, which altogether leads to the dimension $12600 \times 2500$ for the matrix $A$. The exact solution (the Shepp-Logan phantom) is depicted in Figure~\ref{fig_allmethods}. In order to obtain noisy data, we used  $y^\delta := y+\bar{\delta} \norm{y}_2  r$, where $r$ is a randomly generated, normed vector, and $\bar{\delta}$ denotes the relative noise level.

The implementation of the methods introduced in Section~\ref{section_minimization} was done in a straightforward way by using their infinite-dimensional matrix representations but for the now finite dimensional matrices. The iterations were stopped using the discrepancy principle \eqref{discrepancy_nonlinear} with the choice $\tau = 1.1$ for all methods. For the approximation parameter $\eps$ in the definition of $\Jade$, we have used the choice $\eps= 10^{-4} \delta$, which is conforming with the theory developed above. The stepsize $\omega$ in ISTA and FISTA was chosen as a constant based on the norm of $A$, and for the gradient descent method \eqref{gradient_descent}, the stepsizes $\omega_n$ were chosen via the Armijo rule. In the Levenberg-Marquardt method \eqref{Levenberg_Marquardt}, we chose $\alpha_n = 0.6^n \delta$, which is a  sequence tending to $0$ in accordance with the convergence theory. All computations were carried out in Matlab on a desktop computer with an Intel Xeon E5-1650 processor with 3.20GHz and 16 GB RAM.

\subsection{Numerical Results}

In this section, we present the results of applying the iterative methods introduced in Section~\ref{section_minimization} to the tomography problem described above.

In the following, we present reconstruction results for different noise levels $\bar{\delta}$, which is directly related to the signal-to-noise ratio (SNR) by 
	\begin{equation*}
		\bar{\delta} 
		=
		\frac{\norm{y-\yd}}{\norm{y}} 
		\approx
		\frac{\norm{y-\yd}}{\norm{\yd}} 
		=
		\text{SNR}^{-1} \,.
\end{equation*}
The first results, which are related to the computational efficiency of the different methods, are presented in Figure~\ref{fig_comparison}. One can clearly see that regardless of the noise level $\bar{\delta}$, the Newton method and the Levenberg-Marquardt method outperform the gradient based methods, both in terms of computation time and number of iterations $n_*$ required to meet the discrepancy principle. Furthermore, as was to be expected, FISTA also performs much better than both ISTA and the gradient descent method. Note also that with the Levenberg-Marquardt and the Newton method, one can satsify the discrepancy principle also for very small noise levels, which becomes infeasible for the other methods due to the too large runtime which would be required for that.

\begin{figure}[h!]
	\includegraphics[width=0.48\textwidth]{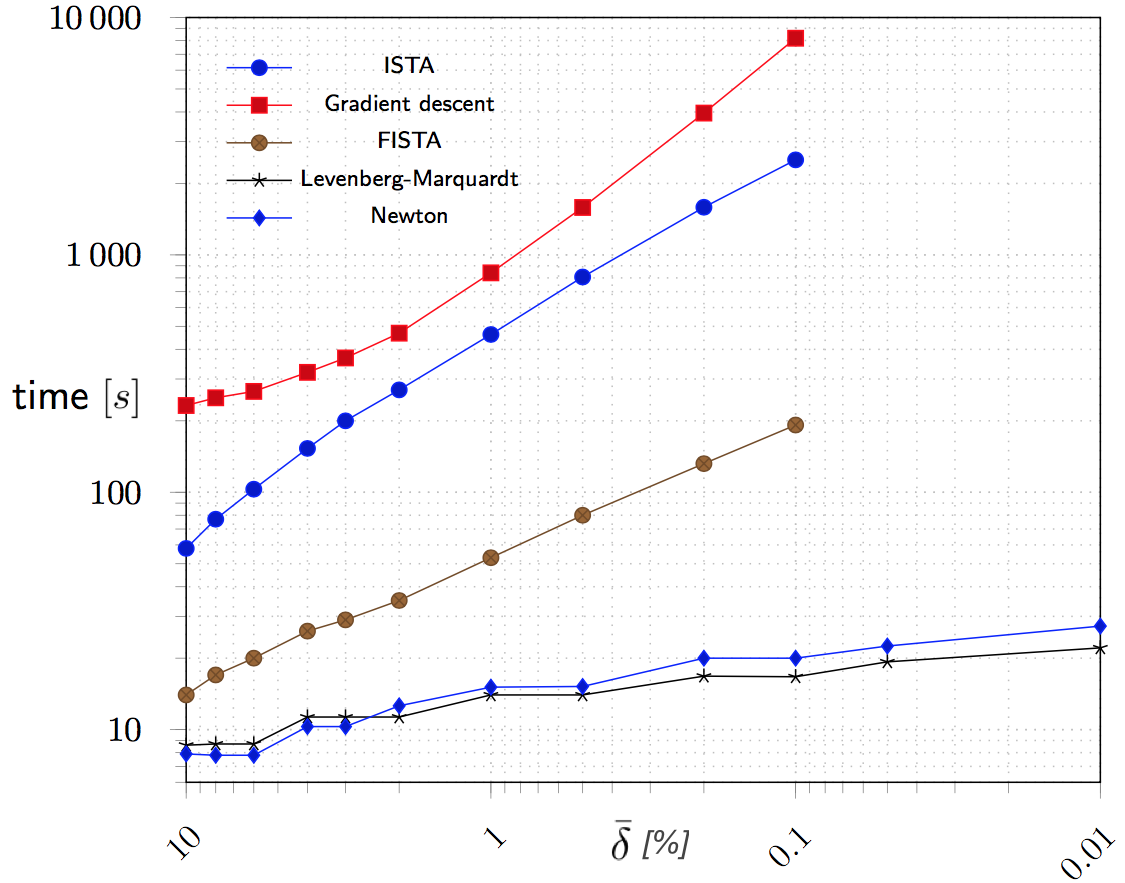} 
	\quad
	\includegraphics[width=0.48\textwidth]{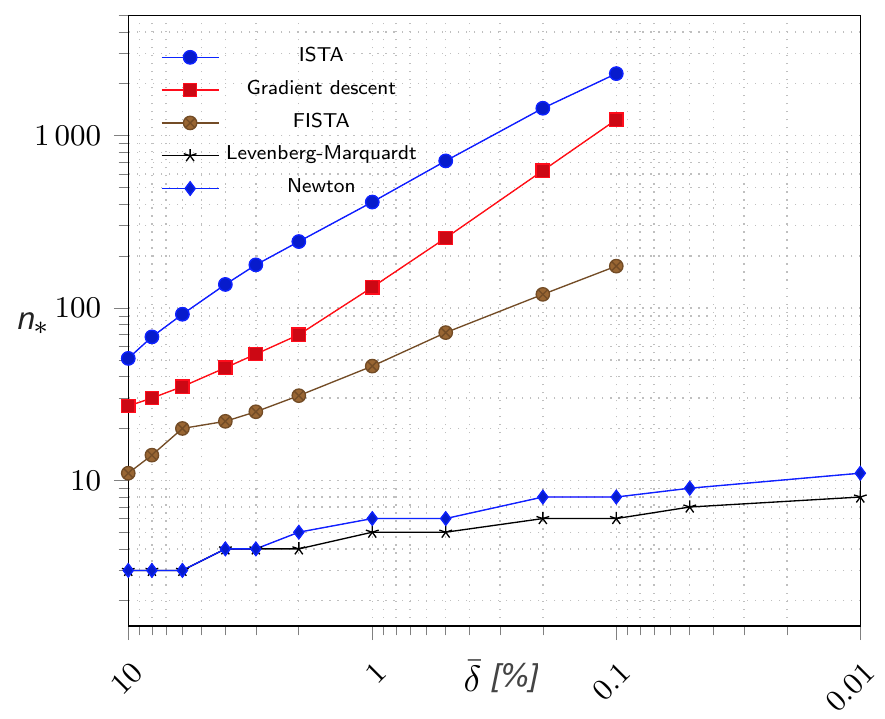} 
	\caption{Elapsed time (left) and number of iterations (right) required for meeting the stopping criterion versus different noise levels, for the considered minimization methods.}
	\label{fig_comparison}
\end{figure}

The results depicted in Figure~\ref{fig_relative_error} show that not only do the Levenberg-Marquardt and the Newton method require less iterations and computation time to satisfy the discrepancy principle, the resulting approximations also have a comparable and even somewhat smaller relative error than for the gradient based methods. This is of course partly due to the fact that each iteration step of those methods is `larger' than in the other methods, which nevertheless turns out to be an advantage in our case. The resulting approximate solutions for $10\%$ relative noise are shown in Figure~\ref{fig_allmethods}. The higher quality of the solutions obtained by the Levenberg-Marquardt and the Newton method is apparent.

\begin{figure}[h!]
	\centering \includegraphics[scale=0.68]{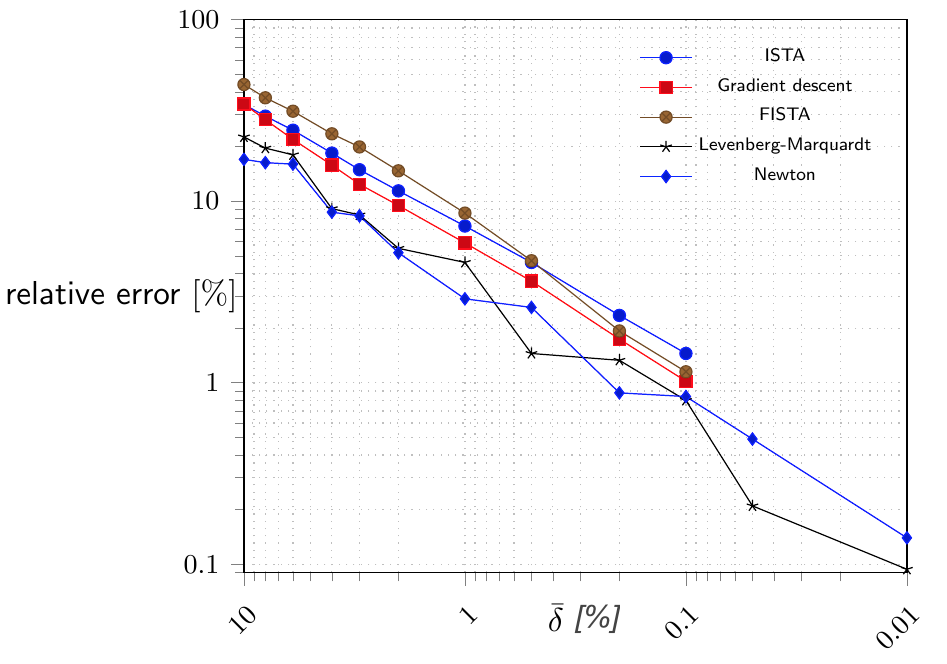} 
	\caption{Relative error $\norm{x_{n_*} - x^\dagger}/\norm{x^\dagger}$ in percent versus different noise levels.}
	\label{fig_relative_error}
\end{figure}

\begin{figure}[h!]
	\centering \includegraphics[width=\textwidth]{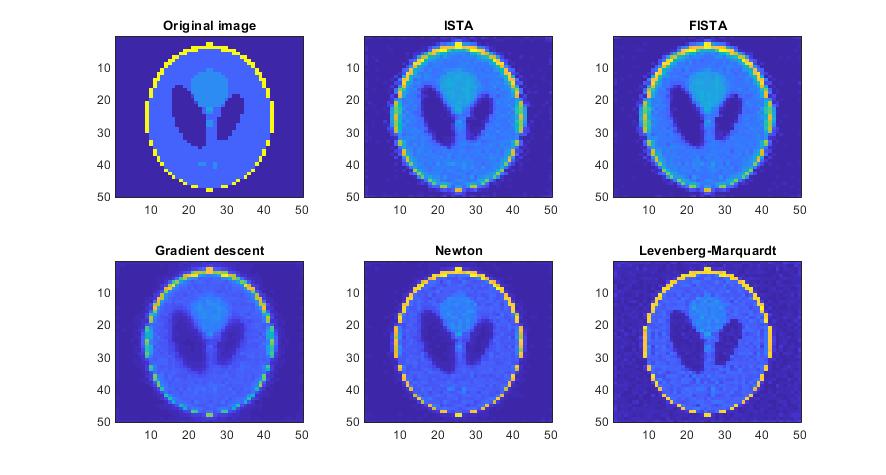} 
	\caption{Exact solution and reconstructions for the noise level $\bar{\delta}=10\%$.}
	\label{fig_allmethods}
\end{figure}

\section{Conclusion}\label{sect_conclusion}

In this paper, we presented a minimization approach for a Tikhonov functional with $\lo$ penalty for the solution of linear inverse problems with sparsity constraints. The employed approximate transformation approach based on a Nemskii operator was mathematically analysed within the framework of ill-posed problems, and the fact that the resulting transformed functional is twice continuously Fr\'echet differentiable served as a basis for the construction of an effective minimization algorithm using Newton's method. Numerical example problems based on the medical imaging problem of computerized tomography demonstrated the usefulness of the proposed approach.

\section{Support}

The authors were funded by the Austrian Science Fund (FWF): F6805-N36.

\bibliographystyle{plain}
{\footnotesize
\bibliography{mybib}
}

\end{document}